\documentclass[a4paper,11pt]{article}
 \usepackage[english]{babel}
 \usepackage{graphicx}
 \usepackage[utf8]{inputenc}
 \usepackage[T1]{fontenc}
 \usepackage{lmodern}
 \usepackage[normalem]{ulem}
 \usepackage{verbatim}
 \usepackage{bbm}
 \usepackage{ntheorem}
 \usepackage{stmaryrd}
 \usepackage{amsmath}
\usepackage{dsfont}
 \usepackage{amssymb}
\usepackage{hyperref}

\usepackage{pgf,tikz}
\usepackage{mathrsfs}
\usetikzlibrary{arrows}

\usepackage{color}

\theoremheaderfont{\scshape}
\theoremseparator{.}
\newtheorem{theorem}{Theorem}[section]
\newtheorem{thm}[theorem]{Theorem}

\newtheorem{fact}[theorem]{Fact}
\newtheorem{coro}[theorem]{Corollary}
\newtheorem{lem}[theorem]{Lemma}
\newtheorem{quest}[theorem]{Question}
\newtheorem{problem}[theorem]{Problem}

\newtheorem{prop}[theorem]{Proposition}

\newenvironment{exam}[1][Example.]{\begin{trivlist}
\item[\hskip \labelsep {\textit{#1}}]}{\end{trivlist}}

\def\sqw{\hbox{\rlap{\leavevmode\raise.3ex\hbox{$\sqcap$}}$%
\sqcup$}}

\newcommand{\N}{\ensuremath{{\mathbb N}}}

\newcommand{\Z}{\ensuremath{\mathbb Z}}

\newcommand{\defini}{\textbf}

\newcommand{\Cay}{\ensuremath{\mathsf{Cay}}}

\newcommand{\R}{\ensuremath{\mathbb R}}

\newcommand{\zigzagfree}{``zigzag-free''}

\newcommand{\dom}{\ensuremath{\mathsf{dom}}}

\newcommand{\Geod}{\mathsf{Geod}}

\newcommand{\diam}{\mathsf{diam}}

\newcommand{\fleche}{\rightarrow}

\newcommand{\saut}{\vspace{0.25cm}}

\newenvironment{proof}{  
    \vspace*{-.4em}  {\it Proof.}%
}{
    \hfill\sqw\vspace*{.5em}
}

\author{S\'ebastien \sc{Martineau}\footnote{E-mail: sebastien.martineau@weizmann.ac.il.}\\
\\
{\it The Weizmann Institute of Science}}
\title{Locally infinite graphs and symmetries}
\begin{document}
\maketitle

\begin{abstract}
When one studies geometric properties of graphs, local finiteness is a common implicit assumption, and that of transitivity a frequent explicit one. By compactness arguments, local finiteness guarantees several regularity properties. It is generally easy to find counterexamples to such regularity results when the assumption of local finiteness is dropped. The present work focuses on the following problem: determining whether these regularity properties still hold when local finiteness is replaced by an assumption of transitivity.

After recalling the locally finite situation, we show that there are Cayley graphs of  $\bigoplus_{n\geq2} \Z/n\Z$ and $\Z$ (with infinite generating systems) that have infinite diameter but do not contain any infinite geodesic ray. We also introduce a notion of generalised diameter. The generalised diameter of a graph is either an ordinal or $\infty$ and captures the extension properties of geodesic paths. It is a finite ordinal if and only if the usual diameter is finite, and in that case the two notions agree. Besides, the generalised diameter is $\infty$ if and only if the considered graph contains an infinite geodesic ray. We show that there exist Cayley graphs of abelian groups of arbitrary generalised diameter.

Finally, we build Cayley graphs of abelian groups that have isomorphic balls of radius $n$ for every $n$ but are not globally isomorphic. This enables us to construct a \emph{non}-transitive graph such that for every $n$ and any vertices $u$ and $v$, the ball of centre $u$ and radius $n$ is isomorphic to that of centre $v$ and same radius.
\end{abstract}

\section{Introduction}

Graphs are fundamental objects in discrete mathematics. They may encode various concepts, such as constraints or geometry. This paper will focus on the geometric perspective. From this point of view, it is common to work under the assumption of local finiteness: a graph is \defini{locally finite} if each of its vertices has finitely many neighbours. Another important notion is that of transitive graphs: a graph is \defini{(vertex-)transitive} (or \defini{homogeneous}) if its automorphism group acts transitively on its vertex-set. See e.g.\ \cite{benjaministflour}.

It is well-known that local finiteness implies some form of compactness, which guarantees that the geometric study of locally finite graphs is in some sense well-behaved. This paper aims at showing that the situation gets wilder when the local finiteness assumption is dropped, \emph{even if} one assumes that the graphs under study are transitive (or weakly transitive, see p.~\pageref{defweaktrans}).

In Section~\ref{rappels}, for completeness and in order to fix conventions, we recall the vocabulary of graph theory that we will need.
In Section~\ref{locfin}, we recall the locally finite situation.
In Section~\ref{endintro}, after a brief exposition of our motivations, we state our questions and results. These results are established in the subsequent sections.

\subsection{Vocabulary of graph theory}

\label{rappels}


\newcommand{\chemin}{\kappa}

{\small
We denote by $\mathbb N$ the set of all non-negative integers and by $\mathbb{N}^\star$ that of positive integers.
}

\vspace{0.2cm}

Given a set $X$, denote by $X \choose 2$ the set whose elements are the subsets of $X$ that contain exactly 2 elements. A \defini{``graph''} is an ordered pair $\mathcal{G}=(V,E)$, where $E\subset {V\choose 2}$. The elements of $V$ are called the \defini{vertices} of $\mathcal{G}$, and the elements of $E$ are called the \defini{edges} of $\mathcal{G}$. Two vertices $u$ and $v$ are said to be \defini{adjacent} (or \defini{neighbours}) if $\{u,v\}$ is an edge. The elements of an edge are called its \defini{endpoints}. 
A \defini{path} is a map $\chemin:I\to V$ such that
\begin{itemize}
\item $I$ is a non-empty subset of $\Z$;
\item $I$ is an ``integer-interval'': $\forall m,n\in I~, \forall k \in \Z,~m\leq k \leq n \implies k\in I$;
\item $\forall n\in I,~n+1\in I \implies \{\chemin(n),\chemin(n+1)\}\in E$.
\end{itemize}
The \defini{length} of $\chemin : I \to V$ is $``\sup I - \inf I\text{''} \in \N\cup\{\infty\}$.
If this length is finite, then the path is said to be \defini{finite}, and it is said to \defini{connect} $\chemin(\min I)$ and $\chemin(\max I)$. A ``graph'' is \defini{connected} if any two vertices are connected by some finite path. In this paper, a \defini{graph} is a connected ``graph'' with a non-empty vertex set.
A graph is \defini{locally finite} if every vertex has finitely many neighbours; this number need not be bounded.
A \defini{rooted graph} is the data of a graph together with that of a vertex of this graph, which is referred to as the \defini{root}.

Given a graph $\mathcal{G}=(V,E)$, one defines a distance on $V$ --- the \defini{graph distance} --- by setting $d(u,v)$ to be the minimal length of a finite path connecting $u$ and $v$. Given $r\in \R_+$ and $u \in V$, the \defini{ball} $B:=B_\mathcal{G}(u,r)$ of radius $r$ and with centre $u$ is the set of the elements $v$ satisfying $d(u,v)\leq r$. It can be seen as a graph by setting its edge-set to be $E\cap {B\choose 2}$. It can be seen as a rooted graph by rooting it at $u$.
The \defini{diameter} of $\mathcal{G}=(V,E)$ is $\sup \{d(u,v):u,v\in V\}\in \N\cup\{\infty\}$.

A \defini{(graph) isomorphism} from a graph $\mathcal{G}=(V,E)$ to a graph $\mathcal{G}'=(V',E')$ is a bijection $\varphi : V\to V'$ such that
$$
\forall u,v\in V,~\{u,v\}\in E \iff \{\varphi(u),\varphi(v)\}\in E'.
$$
Two graphs are \defini{isomorphic} if there is an isomorphism from the first one to the second one. This defines an equivalence relation on the class of graphs. A \defini{(graph) automorphism} of a graph is an isomorphism from itself to itself. These notions extend to rooted graphs by adding the condition that $\varphi$ maps the root of $\mathcal{G}$ to that of $\mathcal{G}'$. Graph automorphisms of a graph form a group under composition, which is called the \defini{automorphism group} of this graph.

A graph $\mathcal G$ is said to be \defini{(vertex-)transitive} (or \defini{homogeneous}) if the natural action of its automorphism group on its vertex-set is transitive, i.e.\ if for any vertices $u$ and $v$, there is a graph automorphism $\varphi$ of $\mathcal G$ such that $\varphi(u)=v$. One says that $\mathcal G$ is \defini{quasi-transitive} if its automorphism group acts on its vertices with finitely many orbits. Given a group $G$ and a generating subset $S$ of $G$, the \defini{Cayley graph} $\Cay(G,S)$ associated with $(G,S)$ is the ``graph'' defined by taking $G$ to be the vertex-set and by declaring two \emph{distinct} vertices $g$ and $h$ to be adjacent if and only if $g^{-1}h\in S\cup S^{-1}$. This ``graph'' is actually a transitive \emph{graph}. The action of $G$ on itself by left-multiplication is free, transitive, and by graph automorphisms of $\Cay(G,S)$. Notice that contrary to many authors, we do \emph{not} assume $S$ to be finite.

A path $\chemin : I \to V$ is \defini{geodesic} if it satisfies
$$
\forall m,n \in I,~d(\chemin(m),\chemin(n))=|m-n|.
$$ A path is \defini{infinite} if its length is infinite. A path is \defini{bi-infinite} if it satisfies $I=\Z$.
A \defini{ray} (or \defini{infinite ray}) is a path $\chemin : \N \to V$. A path $\chemin$ \defini{starts at} some vertex $u$ if $\chemin(0)$ is well-defined and equal to $u$.


\newcommand{\undef}{\textsf{``undefined''}}

Given a graph $\mathcal{G}=(V,E)$, we see the set of its paths as a subset of $(V\cup \{\undef\})^\Z$. Endowing $V\cup \{\undef\}$ with the discrete topology, the product topology on $(V\cup \{\undef\})^\Z$ induces a topology on the space of the paths of $\mathcal{G}$. This will be the only topology we will consider on this space.

\subsection{Regularity for locally finite graphs}

\label{locfin}

Even though some of the following results are seldom stated in this way, no result of Section~\ref{locfin} can be considered to be new. However, I believe it is useful to present these results and their proofs in order to provide context for our work in the non locally finite setup. 

\subsubsection{Infinite geodesic paths}

The power of local finiteness comes from the following observation:

\begin{fact}\label{observa}
In a locally finite rooted graph, the set consisting of the paths starting at the root is compact.
\end{fact}

\begin{small}
\begin{proof}
The set $\mathfrak P$ of the paths in $\mathcal{G}=(V,E)$ that start at $o$ is a subset of $X := \prod_{k\in \Z} B(o,|k|)\cup\{\undef\}$, endowed with the product of discrete topologies. Since $\mathcal{G}$ is locally finite, every $B(o,|k|)$ is finite hence compact. By Tychonoff's Theorem\footnote{No axiom of choice is needed here, as the product is countable and the factors are metrisable. The reader may prefer to proceed by diagonal extraction instead of resorting to Tychonoff's Theorem.}, $X$ is compact. For an element $x$ of $X$, being a path is the conjunction over $(k, l,m)$ in $\Z^2$ such that $k<l<m$ of the following closed conditions:
\begin{itemize}
\item $x(l)=\undef\implies (x(k)=\undef\text{ or }x(m)=\undef)$;
\item $\{x(k),x(k+1)\}\cap\{\undef\}=\varnothing\implies \{x(k),x(k+1)\}\in E$.
\end{itemize}
As a result, the set $\mathfrak P$ is closed in the compact space $X$, hence compact.
\end{proof}
\end{small}

The following result, which is useful in the locally finite framework, does not require local finiteness. We leave its proof to the reader.

\begin{fact}
\label{closedfact}
In any rooted graph, the geodesic paths starting at the root form a closed subset of the space of paths.
\end{fact}

\begin{coro}\label{georay}
Every locally finite graph of infinite diameter contains an infinite geodesic path.
\end{coro}

\begin{small}
\begin{proof}
Let $\mathcal G$ be a locally finite graph of infinite diameter. Fix $o$ a vertex of $\mathcal G$. By assumption, we can find a sequence $(\gamma_n)$ of finite geodesic paths that start at $o$ and such that the length of $\gamma_n$ tends to infinity when $n$ goes to infinity. As $\mathcal G$ is locally finite, by Fact~\ref{observa}, this sequence $(\gamma_n)$ has at least one accumulation point in the space of paths: take $\gamma_\infty$ to be such an accumulation point. By Fact~\ref{closedfact}, $\gamma_\infty$ is actually geodesic. As $\gamma_\infty$ is infinite, Corollary~\ref{georay} is established.
\end{proof}
\end{small}

Under an additional assumption of transitivity, we will prove better: see Corollary~\ref{biinf}. Once again, we will use a fact that does not require local finiteness --- and actually not even transitivity.

\begin{fact}\label{prelem}
In any graph, the geodesic paths form a closed subset of the space of paths, and this subset is invariant by reparametrisation\footnote{i.e.\ by $\gamma\mapsto (\gamma(ak+b))_k$ for any $a\in\{-1,1\}$ and $b\in \Z$} and by graph automorphisms.
\end{fact}

By using Fact~\ref{observa}, we can prove the following lemma.

\begin{lem}
\label{closedlemma}
Let $\mathfrak{P}$ denote a set of paths of some transitive graph $\mathcal{G}$. Assume that the following conditions hold:
\begin{enumerate}
\item \label{condinv} the automorphism group of $\mathcal{G}$ leaves $\mathfrak{P}$ globally invariant;
\item \label{condparam} if $k_0\in \Z$ and $\chemin \in \mathfrak{P}$, then $k\mapsto \chemin (k+k_0)$ belongs to $\mathfrak{P}$;
\item \label{condclos} $\mathfrak{P}$ is closed;
\item \label{condlong} $\mathfrak{P}$ contains paths of arbitrarily large length;
\item \label{locfi} $\mathcal{G}$ is locally finite. \label{locfi}
\end{enumerate}
Then $\mathfrak{P}$ contains a bi-infinite path.
\end{lem}

\begin{small}
\begin{proof}
Let $(\chemin_n)$ denote a sequence of elements of $\mathfrak{P}$ such that the length of $\chemin_n$ is at least $n$. Such a sequence exists by Condition~\ref{condlong}. Let $o$ denote a vertex of $\mathcal{G}$. As $\mathcal{G}$ is transitive, by Conditions~\ref{condinv}, \ref{condparam} and \ref{condlong}, one may assume that for every $n$, the domain of $\chemin_n$ contains $-\lfloor n/2\rfloor$ and $\lfloor n/2\rfloor$ and $\chemin_n(0)=o$. By Condition~\ref{condclos}, Condition~\ref{locfi} and Fact~\ref{observa}, the elements of $\mathfrak{P}$ that start at $o$ form a compact set. The sequence $(\chemin_n)$ thus admits an accumulation point in $\mathfrak{P}$, which must be bi-infinite.

\end{proof}
\end{small}

From Fact~\ref{prelem} and Lemma~\ref{closedlemma}, it results that:

\begin{coro}\label{biinf}
Every transitive locally finite graph of infinite diameter contains a bi-infinite geodesic path.
\end{coro}

Of course, this corollary is false without the transitivity assumption: consider a one-sided infinite ray.

\subsubsection{Local topology is Hausdorff}

\newcommand{\Rootlf}{\mathfrak{G}^{\dagger}_{\mathsf{lf}}}

Local finiteness also plays an important role in the investigation of the so-called local topology. Let us denote by $\Rootlf$ the set consisting in the isomorphism classes of locally finite rooted graphs. Notice that $\Rootlf$ can indeed be realised as a set, as every locally finite graph is isomorphic to a graph the vertex-set of which is a subset of $\N$. By abuse of language, we may, whenever this is not harmful, identify a locally finite rooted graph with its isomorphism class.

Given two rooted graphs $(\mathcal{G},o)$ and $(\mathcal{G}',o')$, we set $d((\mathcal{G},o),(\mathcal{G}',o')):=2^{-K}$, where
$$
K:=\sup\{n\in \N : B_\mathcal{G}(o,n)\simeq B_{\mathcal{G}'}(o',n)\}\in \N\cup\{\infty\}
$$
and the isomorphism between the balls is taken relative to their structure of rooted graphs. The following proposition is well-known.

\begin{prop}\label{hauss}
The map $d$ defines a distance on $\Rootlf$.
\end{prop}

\begin{small}
\begin{proof}
The most interesting point to check is that if $d((\mathcal{G},o),(\mathcal{G}',o'))=0$ then $(\mathcal{G},o)$ and $(\mathcal{G}',o')$ are isomorphic. We will only check this.

Let $(\mathcal{G},o)=((V,E),o)$ and $(\mathcal{G}',o')=((V',E'),o')$ be two rooted graphs such that for every $n$, $B_\mathcal{G}(o,n)$ and $B_{\mathcal{G}'}(o',n)$ are isomorphic as rooted graphs. For every $n$, one can thus fix some isomorphism $\varphi_n$ from $B_\mathcal{G}(o,n)$ to $B_{\mathcal{G}'}(o',n)$. Each $\varphi_n$ can be seen as a map from $V$ to $V'\cup\{\undef\}$, and more precisely as an element of the space $\prod_{v\in V} \left(B_{\mathcal{G}'}(o',d(o,v))\cup\{\undef\}\right)$ --- which is compact by local finiteness of $\mathcal{G}'$. Any accumulation point of the sequence $(\varphi_n)$ is an isomorphism from $\mathcal G$ to $\mathcal{G}'$.
\end{proof}
\end{small}

This distance defines the so-called \defini{local topology} --- see  \cite{babai91, bslimits, diestelleader}. It is very useful in statistical mechanics: see e.g.\ \cite{asuipt, bnp, gllocality1, grimmettli2, cantor, abelianperco}. It is also important in the topic of soficity, which is an important notion at the interplay between group theory and ergodic theory \cite{soficweiss}. The reader is referred to \cite{chaguir, cgp} for some interaction with group theory (via the notion of marked groups \cite{chabautylimit, champetier, grigor85}), and to Section 10 in \cite{aldouslyons} for some graph-theoretic/probabilistic counterpart of soficity.

Say that a graph $\mathcal{G}$ is \defini{weakly transitive}\label{defweaktrans} if for any vertices $o$ and $o'$ and every positive integer $n$, the graphs $B_{\mathcal G}(o,n)$ and $B_{\mathcal G}(o',n)$ are isomorphic as rooted graphs. This definition applies to graphs that may not be locally finite.

\begin{coro}\label{deucoro}
A locally finite graph is transitive if and only if it is weakly transitive.
\end{coro}

\begin{small}
\begin{proof}
Every transitive graph is weakly transitive. Let thus $\mathcal{G}$ denote a locally finite weakly transitive graph. Let $o$ and $o'$ denote two vertices of $\mathcal G$. By weak transitivity of $\mathcal G$, we have $d((\mathcal{G},o),(\mathcal{G},o'))=0$. By Proposition~\ref{hauss}, $(\mathcal{G},o)$ and $(\mathcal{G},o')$ are isomorphic, which means that there is an automorphism of $\mathcal G$ that maps $o$ to $o'$.
\end{proof}
\end{small}

Given a weakly transitive graph, for every $k$, one can speak --- up to isomorphism as a rooted graph --- of its ball of radius $k$.
Say that two weakly transitive graphs are \defini{weakly isomorphic} if for every positive integer $k$, their balls of radius $k$ are isomorphic the one to the other. This definition applies to graphs that may not be locally finite.

The following statement results from Proposition~\ref{hauss}.

\begin{coro}\label{uncoro}
For locally finite (weakly) transitive graphs, weak isomorphism is equivalent to isomorphism.
\end{coro}

\subsection{Beyond the locally finite case}

\label{endintro}

Let us consider the corresponding situations in the non locally finite world. In Section~\ref{endintro}, I will use the word ``fact'' for results whose essence is classical. Except for facts, proofs will be deferred to subsequent sections.

\saut

Before getting to precise statements, I want to explain why such a study is legitimate. Apart from the exploration of the limits of regularity properties for their own sake and my interest for the constructions and questions it leads to, I would like to mention the following motivation: the study of \emph{geometric} properties of some \emph{quasi-transitive} non locally finite graphs plays a true role in the main body of mathematics. For instance, in the study of surfaces, the \emph{curve graph} of an orientable surface $S$ \begin{small}with genus $g$ and $m$ punctures satisfying $3g+m\geq 5$\end{small}, and more specifically its hyperbolicity, plays a major role in the study of the mapping class group of $S$ and its Teichmüller space --- see \cite{mamin1, mamin2}. \begin{small}The hyperbolicity of this graph led to the resolution of Thurston's ending lamination conjecture -- see \cite{lamin}.\end{small}
Likewise, in group theory, the hyperbolicity of the graph of free factors and that of free splittings of the free group with $n\geq 3$ generators $\mathbf{F}_n$ is important in the study of the exterior automorphisms of $\mathbf{F}_n$.

\subsubsection{The study of geodesic paths}

\label{subsubgeod}

Fact~\ref{observa} is \emph{always} false in the non locally finite case.

\begin{fact}
In a non locally finite rooted graph, the set of the paths that start at the root is never compact.
\end{fact}

\begin{small}
\begin{proof}
Let $\mathcal{G}=(V,E)$ denote a graph that is not locally finite, and let $o$ be a vertex of $\mathcal G$. As $\mathcal{G}$ is not locally finite, there is a vertex of $\mathcal G$ with infinitely many neighbours: let $o'$ denote such a vertex. As $\mathcal G$ is connected, there is path $\gamma : \{0,\dots, d(o,o')\} \to V$ such that $\gamma(0)=o$ and $\gamma(d(o,o'))=o'$; fix such a path $\gamma$. As $o'$ has infinitely many neighbours, one can find an injective sequence of vertices $(v_n)_{n\in \N}$ such that for every $n$, $v_n$ is a neighbour of $o'$. For $n\in \N$, define $\chemin_n:\{0,\dots,d(o,o')+1\}\to V$ as agreeing with $\gamma$ where $\gamma$ is defined and taking value $v_n$ at $d(o,o')+1$. For every $n$, $\chemin_n$ is a path. The sequence $(\chemin_n)$ has no accumulation point.
\end{proof}
\end{small}

Corollary~\ref{georay} is false in the non locally finite case.

\begin{fact}\label{cexdiam}
There is a graph of infinite diameter that does not admit any geodesic ray.
\end{fact}

\begin{small}
\begin{proof}
Let $\mathcal{G}$ be the graph defined as follows, and depicted on Figure~\ref{hand}. The vertex-set is $V:=\{(k,n)\in \N^2:1\leq k \leq n\}\cup\{0\}$. The edge-set is defined to be
$$
\{\{(k,n),(l,n)\}:1\leq k\leq n,~ 1\leq l \leq n,~ |k-l|=1\}\cup\{\{0,(1,n)\}:n\geq 1\}.
$$
This ``graph'' is a graph of infinite diameter with no geodesic ray.

\begin{figure}[h!]
\centering

\definecolor{qqqqff}{rgb}{0.,0.,1.}
\begin{tikzpicture}[line cap=round,line join=round,>=triangle 45,x=0.6446708784256676cm,y=0.6446708784256676cm]
\clip(4.204541897368076,4.100379940939294) rectangle (23.439162872153982,13.69664598525259);
\draw (13.68,12.94)-- (15.88,12.94);
\draw (13.68,12.94)-- (15.88,10.74);
\draw (13.68,12.94)-- (15.88,8.54);
\draw (13.68,12.94)-- (15.88,6.34);
\draw (15.88,10.74)-- (18.08,10.74);
\draw (15.88,8.54)-- (18.08,8.54);
\draw (18.08,8.54)-- (20.28,8.54);
\draw (15.88,6.34)-- (18.08,6.34);
\draw (18.08,6.34)-- (20.28,6.34);
\draw (20.28,6.34)-- (22.48,6.34);
\draw (6.56,6.32)-- (4.94,7.2);
\draw (6.56,6.32)-- (5.88466109867141,8.035435037636894);
\draw (6.56,6.32)-- (6.832856874104015,8.14327977179976);
\draw (6.56,6.32)-- (7.44,7.94);
\draw (5.88466109867141,8.035435037636894)-- (5.2093221973428205,9.750870075273788);
\draw (6.832856874104015,8.14327977179976)-- (7.105713748208031,9.96655954359952);
\draw (7.44,7.94)-- (8.32,9.56);
\draw (8.32,9.56)-- (9.2,11.18);
\draw (9.2,11.18)-- (10.08,12.8);
\draw (7.105713748208031,9.96655954359952)-- (7.378570622312047,11.789839315399279);
\draw (15.63,6.289002021221274) node[anchor=north west] {$\rotatebox{90.0}{ \text{ ... }  }$};
\draw (17.672985468334137,4.7) node[anchor=north west] {$\rotatebox{0.0}{ \text{ ... }  }$};
\draw (14.35,8.407475957904472) node[anchor=north west] {$\rotatebox{90.0}{ \text{ ... }  }$};
\draw (7.29,8.15) node[anchor=north west] {$\rotatebox{-35.0}{ \text{ ... }  }$};
\begin{scriptsize}
\draw [fill=qqqqff] (13.68,12.94) circle (2.0pt);
\draw [fill=qqqqff] (15.88,12.94) circle (2.0pt);
\draw [fill=qqqqff] (15.88,10.74) circle (2.0pt);
\draw [fill=qqqqff] (15.88,8.54) circle (2.0pt);
\draw [fill=qqqqff] (18.08,10.74) circle (2.0pt);
\draw [fill=qqqqff] (18.08,8.54) circle (2.0pt);
\draw [fill=qqqqff] (20.28,8.54) circle (2.0pt);
\draw [fill=qqqqff] (15.88,6.34) circle (2.0pt);
\draw [fill=qqqqff] (18.08,6.34) circle (2.0pt);
\draw [fill=qqqqff] (20.28,6.34) circle (2.0pt);
\draw [fill=qqqqff] (22.48,6.34) circle (2.0pt);
\draw [fill=qqqqff] (6.56,6.32) circle (2.0pt);
\draw [fill=qqqqff] (4.94,7.2) circle (2.0pt);
\draw [fill=qqqqff] (5.88466109867141,8.035435037636894) circle (2.0pt);
\draw [fill=qqqqff] (6.832856874104015,8.14327977179976) circle (2.0pt);
\draw [fill=qqqqff] (7.44,7.94) circle (2.0pt);
\draw [fill=qqqqff] (5.2093221973428205,9.750870075273788) circle (2.0pt);
\draw [fill=qqqqff] (7.105713748208031,9.96655954359952) circle (2.0pt);
\draw [fill=qqqqff] (7.378570622312047,11.789839315399279) circle (2.0pt);
\draw [fill=qqqqff] (8.32,9.56) circle (2.0pt);
\draw [fill=qqqqff] (9.2,11.18) circle (2.0pt);
\draw [fill=qqqqff] (10.08,12.8) circle (2.0pt);
\end{scriptsize}
\end{tikzpicture}

\caption{Two ways of picturing the graph used in the proof of Fact~\ref{cexdiam}.}
\label{hand}
\end{figure}
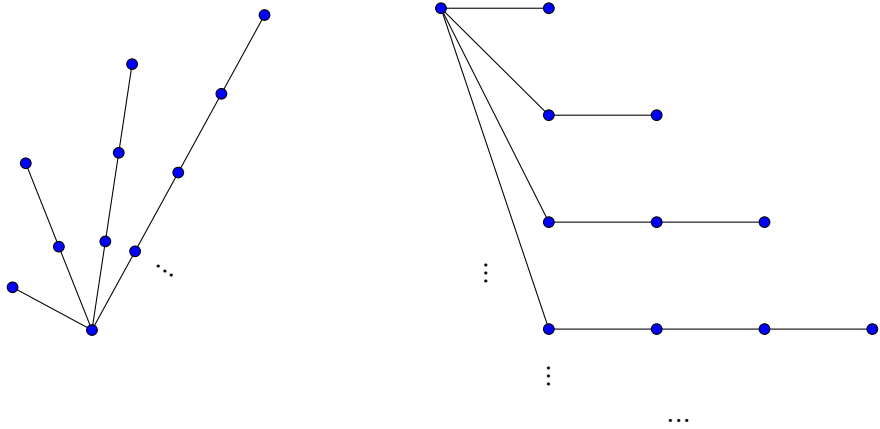
\end{proof}
\end{small}

Corollary~\ref{georay} is also false in the transitive non locally finite case: actually one can even find counterexamples that are Cayley graphs of abelian groups.

\begin{prop}
The groups $\Z$ and $\bigoplus_{n\geq 2} \Z/n\Z$ admit Cayley graphs of infinite diameter that do not contain any geodesic ray.
\end{prop}

We introduce on page \pageref{gendiamdef} the generalised diameter of a graph. The generalised diameter of a graph is either an ordinal, or the symbol $\infty$. It is a finite ordinal if and only if the usual diameter is finite, and in that case the two notions agree. Moreover, the generalised diameter is $\infty$ if and only if the considered graph contains a geodesic ray. 
When this generalised diameter is an infinite ordinal, it captures finely the extension properties of geodesic paths.

\begin{thm}
There are Cayley graphs\begin{small} of abelian groups\end{small} of every generalised diameter.
\end{thm}


As Corollary~\ref{georay} is false without the assumption of local finiteness, it must also be the case of Corollary~\ref{biinf}.
However, one may ask the following question:

\begin{quest}
Does every transitive graph that contains a geodesic ray contain a bi-infinite geodesic path?
\end{quest}

Even though I do not know the answer to this question, we can understand the situation for Cayley graphs of abelian groups (Proposition~\ref{yesabelian} and Proposition~\ref{cexlemma}): if the assumption of local finiteness is replaced by that of being a Cayley graph of an abelian group, then Lemma~\ref{closedlemma} does not hold, but its Corollary~\ref{biinf} remains true.

\begin{prop}\label{yesabelian}
Let $\mathcal{G}$ be a Cayley graph of an abelian group. Assume that $\mathcal{G}$ contains a geodesic ray. Then, $\mathcal{G}$ contains a bi-infinite geodesic ray.
\end{prop}

\begin{prop}\label{cexlemma}
Lemma~\ref{closedlemma} does not hold if Condition~\ref{locfi} is removed, even if $\mathcal{G}$ is assumed to be a Cayley graph of an abelian group and $\mathfrak{P}$ is assumed to be invariant under $\chemin \mapsto (k\mapsto \chemin(-k))$.
\end{prop}

\subsubsection{Weak isomorphy and transitivity}

\label{subsubhauss}

Proposition~\ref{hauss} fails to hold for two reasons in the non locally finite case. First, the space of isomorphism classes of rooted graphs cannot be realised as a set. Second, two rooted graphs $(\mathcal{G},o)$ and $(\mathcal{G}',o')$ may satisfy $B_{\mathcal G}(o,n)\simeq B_{\mathcal{G}'}(o',n)$ in the rooted sense for every $n\in \N$, but still not be isomorphic. Consider the example of the proof of Fact~\ref{cexdiam} rooted at $0$ and the same construction where $n$ is taken in $\N\cup\{\infty\}$ and $k$ in $\N$. ``See'' Figure~\ref{localfig}.

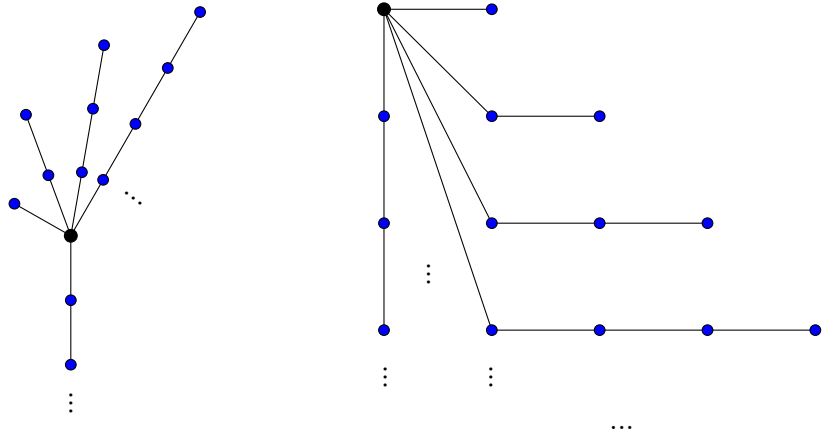
\begin{figure}[h!]
\centering

\definecolor{qqqqff}{rgb}{0.,0.,1.}
\begin{tikzpicture}[line cap=round,line join=round,>=triangle 45,x=0.6446708784256676cm,y=0.6446708784256676cm]
\clip(4.204541897368076,4.100379940939297) rectangle (23.439162872153982,13.696645985252589);
\draw (13.68,12.94)-- (15.88,12.94);
\draw (13.68,12.94)-- (15.88,10.74);
\draw (13.68,12.94)-- (15.88,8.54);
\draw (13.68,12.94)-- (15.88,6.34);
\draw (15.88,10.74)-- (18.08,10.74);
\draw (15.88,8.54)-- (18.08,8.54);
\draw (18.08,8.54)-- (20.28,8.54);
\draw (15.88,6.34)-- (18.08,6.34);
\draw (18.08,6.34)-- (20.28,6.34);
\draw (20.28,6.34)-- (22.48,6.34);
\draw (7.291060215714465,8.281209299426665)-- (6.140630660694447,8.940601849255211);
\draw (7.291060215714465,8.281209299426665)-- (6.833629408840945,9.525815161840942);
\draw (7.291060215714465,8.281209299426665)-- (7.517216447695522,9.587784997413229);
\draw (7.291060215714465,8.281209299426665)-- (7.950452765543011,9.431638854446682);
\draw (6.833629408840945,9.525815161840942)-- (6.376198601967426,10.770421024255219);
\draw (7.517216447695522,9.587784997413229)-- (7.743372679676579,10.894360695399792);
\draw (7.950452765543011,9.431638854446682)-- (8.609845315371556,10.582068409466698);
\draw (8.609845315371556,10.582068409466698)-- (9.269237865200102,11.732497964486715);
\draw (9.269237865200102,11.732497964486715)-- (9.928630415028648,12.882927519506731);
\draw (7.743372679676579,10.894360695399792)-- (7.969528911657636,12.200936393386355);
\draw (15.62000042,6.289002021221276) node[anchor=north west] {$\rotatebox{90.0}{ \text{ ... }  }$};
\draw (17.672985468334137,4.579335436443212) node[anchor=north west] {$\rotatebox{0.0}{ \text{ ... }  }$};
\draw (14.346970862790185,8.407475957904474) node[anchor=north west] {$\rotatebox{90.0}{ \text{ ... }  }$};
\draw (7.81,9.664500042) node[anchor=north west] {$\rotatebox{-35.0}{ \text{ ... }  }$};
\draw (7.0523,5.769905758590294) node[anchor=north west] {$\rotatebox{90.0}{ \text{ ... }  }$};
\draw (7.291060215714465,8.281209299426665)-- (7.291060215714463,6.955205266313288);
\draw (7.291060215714461,5.629201233199911)-- (7.291060215714463,6.955205266313288);
\draw (13.68,12.94)-- (13.68,10.74);
\draw (13.68,10.74)-- (13.68,8.54);
\draw (13.68,8.54)-- (13.68,6.34);
\draw (13.462304253445538,6.289002021221276) node[anchor=north west] {$\rotatebox{90.0}{ \text{ ... }  }$};
\begin{scriptsize}
\draw [fill=black] (13.68,12.94) circle (2.4pt);
\draw [fill=qqqqff] (15.88,12.94) circle (2.0pt);
\draw [fill=qqqqff] (15.88,10.74) circle (2.0pt);
\draw [fill=qqqqff] (15.88,8.54) circle (2.0pt);
\draw [fill=qqqqff] (18.08,10.74) circle (2.0pt);
\draw [fill=qqqqff] (18.08,8.54) circle (2.0pt);
\draw [fill=qqqqff] (20.28,8.54) circle (2.0pt);
\draw [fill=qqqqff] (15.88,6.34) circle (2.0pt);
\draw [fill=qqqqff] (18.08,6.34) circle (2.0pt);
\draw [fill=qqqqff] (20.28,6.34) circle (2.0pt);
\draw [fill=qqqqff] (22.48,6.34) circle (2.0pt);
\draw [fill=black] (7.291060215714465,8.281209299426665) circle (2.4pt);
\draw [fill=qqqqff] (6.140630660694447,8.940601849255211) circle (2.0pt);
\draw [fill=qqqqff] (6.833629408840945,9.525815161840942) circle (2.0pt);
\draw [fill=qqqqff] (7.517216447695522,9.587784997413229) circle (2.0pt);
\draw [fill=qqqqff] (7.950452765543011,9.431638854446682) circle (2.0pt);
\draw [fill=qqqqff] (6.376198601967426,10.770421024255219) circle (2.0pt);
\draw [fill=qqqqff] (7.743372679676579,10.894360695399792) circle (2.0pt);
\draw [fill=qqqqff] (7.969528911657636,12.200936393386355) circle (2.0pt);
\draw [fill=qqqqff] (8.609845315371556,10.582068409466698) circle (2.0pt);
\draw [fill=qqqqff] (9.269237865200102,11.732497964486715) circle (2.0pt);
\draw [fill=qqqqff] (9.928630415028648,12.882927519506731) circle (2.0pt);
\draw [fill=qqqqff] (7.291060215714463,6.955205266313288) circle (2.0pt);
\draw [fill=qqqqff] (7.291060215714461,5.629201233199911) circle (2.0pt);
\draw [fill=qqqqff] (13.68,10.74) circle (2.0pt);
\draw [fill=qqqqff] (13.68,8.54) circle (2.0pt);
\draw [fill=qqqqff] (13.68,6.34) circle (2.0pt);
\end{scriptsize}
\end{tikzpicture}

\caption{Two ways of picturing the same rooted graph.}
\label{localfig}
\end{figure}

One can also find such counterexamples that are Cayley graphs of abelian groups, so that Corollary~\ref{uncoro} does not hold beyond the locally finite case.

\begin{prop}
There are two Cayley graphs of abelian groups that are weakly isomorphic but not isomorphic.
\end{prop}

Corollary~\ref{deucoro} does not hold without local finiteness.

\begin{prop}
There exists a graph that is weakly transitive and weakly isomorphic to a transitive graph, but that is not transitive.
\end{prop}

Here is a natural question:

\begin{quest}
Is there a weakly transitive graph which is not weakly isomorphic to any transitive graph?
\end{quest}

I currently do not know how to solve this question, but I expect such a graph to exist. Such a graph would have the flavour of a Penrose tiling (a ``quasi-periodic'' tiling of the plane made of tiles that cannot tile the plane periodically). Indeed, we would not have tiles but a system of balls\footnote{i.e.\ the data for each radius $k\in \N$ of a rooted graph which is thought of as the prescribed ball of radius $k$ up to isomorphism} that can be nicely put together\footnote{which does not mean that ``they can tile the plane quasi-periodically'' but that ``there is a graph such that, at every vertex, the balls of every radius are isomorphic to the prescribed ones''} but such that this cannot be done ``periodically''\footnote{meaning here that no way of nicely putting together the balls of the system yields a transitive graph}.
This question naturally leads to the following problem, which I cannot solve either.

\begin{problem}
Given $(m,n)\in\N^2$ such that $m\leq n$, determine whether there is a rooted graph $(\mathcal{B},o)$ such that the following two properties hold:
\begin{itemize}
\item there is a graph $\mathcal{G}=(V,E)$ such that, for every $v\in V$, the rooted graphs $(B_{\mathcal G}(v,n),v)$ and $(\mathcal{B},o)$ are isomorphic
\item and there is no \emph{transitive} graph $\mathcal{G}=(V,E)$ such that, for every $v\in V$, the rooted graphs $(B_{\mathcal G}(v,m),v)$ and $(B_{\mathcal B}(o,m),o)$ are isomorphic.
\end{itemize}
\end{problem}

\noindent I cannot answer this problem either in the finite or locally finite settings.

\section{The study of geodesic paths}

Here is an easy example of a transitive graph of infinite diameter that does not contain any geodesic ray.
\begin{prop}
Let $G$ denote the group $\bigoplus_{n\geq 2} \Z/n\Z$ and $S$ be the generating subset $\{g\in G : \forall n,~g_n \in\{-1,0,1\}\}$. The Cayley graph of $G$ relative to $S$ has infinite diameter but does not contain any geodesic ray.
\end{prop}

\begin{proof}
Let $G$ and $S$ be as in the statement of the proposition, and let $\mathcal{G}$ denote the corresponding Cayley graph. The graph $\mathcal{G}$ has infinite diameter, as $g\mapsto g_n$ defines a 1-Lipschitz map onto $\Cay(\Z/n\Z,\{1\})$ and the diameter of $\Cay(\Z/n\Z,\{1\})$ tends to infinity as $n$ goes to infinity.

Now, let $\chemin : \N \to G$ be a path in $\mathcal{G}$ and let us prove that $\chemin$ is not geodesic. Without loss of generality, we may assume that $\chemin(0)=0$. There is some $N\geq 2$ such that $\forall n\geq N,~\chemin(1)_n=0$. Fix such an $N$ and notice that $d(\chemin(0),\chemin(N))\leq N-1$. Indeed, for every integer $n \geq 2$, either $n < N$ and the diameter of $\Cay(\Z/n\Z,\{1\})$ is at most $N-1$, or $n\geq N$ and $\chemin(1)_n=0$ implies that $\chemin(N)_n$ lies in the ball of centre $0$ and radius $N-1$ in $\Cay(\Z/n\Z,\{1\})$.

\end{proof}

We also present the following example for its flavour of additive number theory.

\begin{prop}\label{propz}
Let $S$ denote the (generating) subset of $\Z$ consisting of the elements that can be written as $\sum_{n\geq 1} \varepsilon_n n!$ for some finitely supported element $\varepsilon$ of $\{-1,0,1\}^{\N^\star}$.
The Cayley graph of $\Z$ relative to $S$ has infinite diameter but does not contain any geodesic ray.
\end{prop}

\begin{proof}
Let $S$ be as in the proposition above and let $\mathcal{G}$ denote the Cayley graph of $\Z$ relative to $S$. First, let us prove that $\mathcal{G}$ has infinite diameter. Let $N\geq 2$. Let $\pi$ denote the canonical projection $\Z\to \Z/N!\Z$. The map $\pi$ is $1$-Lipschitz from $\mathcal G$ onto $\Cay(\Z/N!\Z, \pi(S))$. Besides, the identity map of $\Z/N!\Z$ is $(1+2+\dots+(N-1)!)$-Lipschitz, hence $e(N-1)!$-Lipschitz, from $\Cay(\Z/N!\Z, \pi(S))$ to $\Cay(\Z/N!\Z, \{1\})$. As a result, $\mathcal{G}$ has diameter at least $\frac{N!}{2}\times \frac{1}{e(N-1)!}=\frac{N}{2e}$. As $N$ is arbitrary, the diameter of $\mathcal{G}$ must be infinite.

Now, let us assume for contradiction that $\mathcal{G}$ admits a geodesic ray $\gamma : \N \to \Z$.
For every $k \in \N^\star$, fix $\varepsilon^{(k)}$ such that $\Delta(k):=\gamma(k)-\gamma(k-1)=\sum_{n\geq 1} \varepsilon_n^{(k)} n!$.
First, let us show that without loss of generality, we may change $\gamma$ --- and the $\varepsilon^{(k)}$'s --- so that for every $n$ in the support of $\varepsilon^{(1)}$ and every $k\in \N^\star$, we have $\varepsilon^{(k)}_n=\varepsilon^{(1)}_n$.
If there is some $n$ in the support of $\varepsilon^{(1)}$ such that there is some $k_n\in \N^\star$ satisfying $\varepsilon^{(k_n)}_n\not= \varepsilon^{(1)}_n$, then we may change $\gamma$ to have the same initial point, and the same $\Delta(k)$'s except for $\Delta_\text{new}(1):=\Delta_\text{old}(1)-\varepsilon^{(1)}_n n!$ and $\Delta_\text{new}(k_n):=\Delta_\text{old}(k_n)+\varepsilon^{(1)}_n n!$. The new $\gamma$ is still a geodesic path as it is a 1-Lipschitz map from $\N$ to $\mathcal{G}$ that agrees with the geodesic path $\gamma_\text{old}$ on $\{0\}\cup \{k_n,k_n+1,\dots\}$. As a result, if $\gamma$ does not satisfy the desired property, then we may change $(\gamma,(\varepsilon^{(k)}))$ so that $\gamma$ is still a geodesic ray but with an $\varepsilon^{(1)}$ of strictly smaller support. By Fermat's infinite descent, this process must stop at some suitable $\gamma$. Henceforth, we assume that  we work with such a $\gamma$. We may further assume that $\Delta(1)>0$, as $x\mapsto -x$ is a graph automorphism of $\mathcal{G}$.

Let $M$ denote the maximum of the support of $\varepsilon^{(2)}$ --- which is non-empty as $\gamma$ is geodesic. Let $K > \max (M,\Delta(1))$. As $\gamma$ has been suitably modified, one may define a $1$-Lipschitz map $\lambda : \{0,\dots,\frac{K!}{\Delta(1)}\} \to \Cay(\Z,S)$ by setting $\eta(0)=\gamma(0)$ and
$$\forall k \in \{1,\dots,{K!}/{\Delta(1)}\},~\lambda(k)-\lambda(k-1)=\Delta(k)-\Delta(1)+\mathds{1}_{k=2}K!.$$
As $\lambda(1)=\gamma(0)$ and $\lambda({K!}/{\Delta(1)})=\gamma({K!}/{\Delta(1)})$, the fact that $\lambda$ is 1-Lipschitz contradicts the fact that $\gamma$ is geodesic, which ends the proof.
\end{proof}

In order to capture the extension properties of geodesic paths, let us define the generalised diameter of a graph and the generalised radius of a rooted graph.

\vspace{0.3 cm}

Let $\mathcal G$ be a graph. We denote by $\Geod(\mathcal{G})$ the set of the geodesic paths of $\mathcal{G}$ that are well-defined at 0 but not at $-1$. We define a partial order on $\Geod(\mathcal{G})$ by setting, for every  $\gamma$ and $\kappa$ in $\Geod(\mathcal{G})$,
$$
\gamma \leq \kappa \iff \gamma\text{ extends }\kappa.
$$
The \defini{generalised diameter}\label{gendiamdef} of $\mathcal G$ is the smallest ordinal $\eta$ such that there is a way to assign to each element $\gamma$ of $\Geod(\mathcal{G})$ an ordinal $\xi(\gamma)$ such that:
\begin{itemize}
\item for every $\gamma \in \Geod(\mathcal{G})$, $\xi(\gamma)\leq \eta$,
\item and for every $\gamma$ and $\kappa$ in $\Geod(\mathcal{G})$, if $\gamma < \kappa$, then $\xi(\gamma) < \xi(\kappa)$.
\end{itemize}
If some $\eta$ satisfies these conditions, there is a smallest one satisfying them: in that case, the generalised diameter\footnote{which is well-defined and equal to the Krull ordinal of $(\Geod(\mathcal{G}),\leq)$; see Section~2.6 in \cite{bassordinal}} is said to be \defini{transfinite}. If there is no such $\eta$, we define the generalised diameter to be $\infty$, and say that it is \defini{truly infinite}, or \defini{intransfinite}: it is larger than any ordinal number.

If $(\mathcal{G},o)$ is a rooted graph, we set $\Geod(\mathcal{G},o)$ to be the set of the elements of $\Geod(\mathcal{G})$ that start at $o$. By using $\Geod(\mathcal{G},o)$ instead of $\Geod(\mathcal{G})$, one defines the \defini{generalised radius} of $(\mathcal{G},o)$.

Here are a few important remarks.
\begin{itemize}
\item If a graph $\mathcal{G}$ has a vertex $o$ such that the generalised radius of $(\mathcal{G},o)$ is $\infty$, then all its vertices satisfy this property. See Lemma~2.4 of \cite{yadintointon}, which does not rely on any form of local finiteness, only connectedness.
\item The generalised diameter of a graph $\mathcal{G}=(V,E)$ is the supremum over $o\in V$ of the generalised radius of $(\mathcal{G},o)$. In particular, if $\mathcal{G}$ is transitive, then for every vertex $o$ of $\mathcal{G}$, the generalised diameter of $\mathcal{G}$ is equal to the generalised radius of $(\mathcal{G},o)$.
\item The generalised diameter of a graph is finite if and only if its diameter is finite, and in that case the two notions agree.
\item The generalised diameter of a graph is intransfinite if and only if this graph contains a geodesic ray.
\end{itemize}

\begin{exam}
The generalised radius and diameter allow us to capture finely the extension properties of the geodesic paths. For instance, a rooted graph $(\mathcal{G},o)$ has a generalised radius larger than or equal to $\omega 2+5$ if and only if there is a geodesic path $\gamma$ starting at $o$ and of length 5 that can be extended to geodesic paths of arbitrary length that themselves can be extended to geodesic paths of arbitrary length, i.e.
\footnotesize $$
\exists \gamma\in \Geod_5(\mathcal{G},o),~\forall n\geq 5,~\exists \kappa\in \Geod_n(\mathcal{G},o),~\kappa \leq \gamma\text{ \& }\forall m\geq n,~\exists \lambda\in \Geod_m(\mathcal{G},o),~\lambda \leq \kappa$$\normalsize
where $\Geod_n(\mathcal{G},o)$ denotes the set of the elements of $\Geod(\mathcal{G},o)$ of length $n$ and the order $\leq$ has been introduced after the proof of Proposition~\ref{propz}.
\end{exam}

Let us introduce some useful notation. Let $(\mathcal{G},o)$ be a rooted graph. Define by transfinite induction, for every ordinal $\eta$, the set
$$
D_\eta := \{\gamma \in \Geod(\mathcal{G},o): \forall \kappa\in \Geod(\mathcal{G},o),~ \kappa \leq \gamma \implies \exists \xi < \eta,~\kappa\in D_\xi\}.
$$
In words, one can say that we remove from $\Geod(\mathcal{G},o)$ its minimal elements and put them in $D_0$; once this erasure is performed, we remove the minimal elements of the remaining set, and put them, together with the previously erased elements, in $D_1$; etc. The generalised radius of $(\mathcal{G},o)$ is the smallest $\eta$ such that $D_\eta=\Geod(\mathcal{G},o)$ --- which is set to be $\infty$ if no such $\eta$ exists.

If $A$ denotes a set of ordinals, $\sup^+ A$ denotes $\sup \{\eta+1;\eta\in A\}$. It is the unique ordinal $\eta_0$ such that for every ordinal $\eta$, the following equivalence holds:
$$
\eta \geq \eta_0 \iff \forall \xi \in A,~\eta > \xi.
$$
Besides, wet set $\sup^+ A\cup\{\infty\}:=\infty$.

Let $(\mathcal{G},o)$ be a rooted graph. Given two vertices $v$ and $w$, write $v\fleche w$ as an abbreviation for ``$\{v,w\}$ is an edge and $d(o,w)=d(o,v)+1$''. Given a vertex $v$, any two geodesic paths from $o$ to $v$ belong exactly to the same $D_\eta$'s. As for any $v$ there is a geodesic path from $o$ to $v$, one can define the \defini{label} $\ell_v$ of $v$ in $(\mathcal{G},o)$ to be the least $\eta$ such that $\gamma$ belongs to $D_\eta$ --- where $\gamma$ is a geodesic path from $o$ to $v$ and the definition does not depend on the choice of $\gamma$. If no such ordinal exists, set $\ell_v$ to be $\infty$. This labeling is the unique assignment $L$ of an ordinal or $\infty$ to each vertex satisfying the following equation:
$
\forall v,~L_v=\sup^+\{L_w: v\fleche w\}
$.

The generalised radius of a rooted graph $(\mathcal{G},o)$ is equal to $\ell_o$. Recall that if $\mathcal{G}$ is transitive, then the generalised diameter of $\mathcal G$ is the generalised radius of $(\mathcal{G},o)$.

\begin{thm}
For every ordinal $\eta$, there is a Cayley graph of an abelian group the diameter of which is transfinite and equal to $\eta$.
\end{thm}

\begin{proof}
For every ordinal $\eta$, let us define some abelian group $G_\eta$ and a generating subset $S_\eta$ of $G_\eta$. We do so by transfinite induction, by using the following rule:
\begin{itemize}
\item if $\eta$ can be written as $\xi+1$, then set $G_{\eta}:=G_{\xi}\oplus \Z/2\Z$ and $S_{\eta}:=(S_{\xi}\times\{0\})\cup\{(0,1)\}$;
\item otherwise, set $G_\eta := \bigoplus_{\xi<\eta} G_\xi$ and $S_\eta := \{g\in G_\eta:\forall \xi<\eta,~g_{\xi}\in S_{\xi}\}$.
\end{itemize}
Let us prove by transfinite induction that for every $\eta$, the diameter of the Cayley graph $\mathcal{G}_\eta$ of $(G_\eta,S_\eta)$ is transfinite and equal to $\eta$. 
Let $\eta$ be an ordinal number such that for every $\xi<\eta$, the diameter of $\mathcal{G}_{\xi}$ is transfinite and equal to $\xi$. Let us show that the diameter of $\mathcal{G}_{\eta}$ is transfinite and equal to $\eta$.
Denote by $g\mapsto \ell^\eta_g$ the label-map of $(\mathcal{G}_\eta,0)$.

\saut

First, assume that $\eta$ can be written as $\xi+1$. Let us prove that for every $(g,x)\in G_{\eta}=G_{\xi}\oplus \Z/2\Z$, one has $\ell_{(g,x)}^{\eta}=\ell_g^{\xi}+\mathds{1}_{x=0}$. Notice that this implies that $\ell_{(0,0)}^\eta=\ell_0^{\xi}+1=\xi+1=\eta$, i.e.\ that the diameter of $\mathcal{G}_{\eta}$ is transfinite and equal to $\eta$.
To prove the claim, notice that for every $g$ and $h$ in $G_{\xi}$,
\begin{enumerate}
\item \label{un} $(g,0)\fleche (h,0) \iff (g,1)\fleche (h,1) \iff g\fleche h$;
\item \label{deux}$(g,0)\fleche (h,1)\iff g=h$;
\item \label{trois} $(g,1)\not\fleche (h,0)$.
\end{enumerate}
For $(g,x)\in G_\eta$, let $\ell_{(g,x)}:=\ell_g^{\xi}+\mathds{1}_{x=0}$. What we need to prove is that for every $(g,x)\in G_\eta$, one has $\ell_{(g,x)}=\sup^+\{\ell_{(h,y)}:(g,x)\fleche (h,y)\}$.
Let $(g,x)\in G_\eta$.
If $x=1$, then by \ref{un} and \ref{trois}, one has $$\ell_{(g,x)}=\ell^\xi_g=\text{sup}^+\{\ell_h^\xi:g\fleche h\}=\text{sup}^+\{\ell_{(h,y)}:(g,x)\fleche (h,y)\}.$$
If $x=0$, then by \ref{deux}, one has $$\ell_{(g,x)}=\ell^\xi_g+1=\text{sup}^+\{\ell_{(g,1)}\} \leq \text{sup}^+\{\ell_{(h,y)}:(g,x)\fleche (h,y)\}.$$ Still assuming $x=0$, it remains to show that $$\ell_{(g,x)}\geq\text{sup}^+\{\ell_{(h,y)}:(g,x)\fleche (h,y)\}.$$ In other words, let $(h,y)$ be such that $(g,0)\fleche (h,y)$ and let us show that $\ell_{(g,0)}> \ell_{(h,y)}$. If $y=0$, then by \ref{un}, one has $g\fleche h$, so that $\ell_{(g,0)}= \ell^\xi_g+1>\ell^\xi_h+1= \ell_{(h,y)}$. If $y=1$, then by \ref{deux}, one has $g=h$, so that $\ell_{(g,0)}= \ell^\xi_g+1>\ell^\xi_g= \ell_{(h,y)}$.


\saut

Now, assume that $\eta$ is a limit ordinal. Let us prove that for every $g\in G_\eta = \bigoplus_{\xi<\eta} G_\xi$, one has $\ell_{g}^\eta=\sup\{\ell_{g_\xi}^\xi:\xi\in [0,\eta)\text{ and }d(0,g_\xi)=d(0,g)\}$. Notice that this implies that $\ell_0^\eta= \sup \{\ell_{0}^\xi:\xi<\eta\} = \sup \{\xi : \xi < \eta\}=\eta$, i.e.\ that the diameter of $\mathcal{G}_{\eta}$ is transfinite and equal to $\eta$. For $g\in G_\eta$, let $\ell_g:= \sup\{\ell_{g_\xi}^\xi:\xi\in [0,\eta)\text{ and }d(0,g_\xi)=d(0,g)\}$. What we need to prove is that for every $g\in G_\eta$, one has $\ell_{g}=\sup^+\{\ell_h:g\fleche h\}$. Let $g\in G_\eta$. First, let us show that $\ell_g \geq \sup^+\{\ell_h:g\fleche h\}$, i.e.\ let $h\in G_\eta$ be such that $g\fleche h$ and let us prove that $\ell_g >\ell_h$. As $g\fleche h$, the element $h$ cannot be 0. Thus, by definition of $(G_\eta,S_\eta)$, the set $E_h:=\{\xi\in [0,\eta)\text{ and }d(0,h_\xi)=d(0,h)\}$ is finite. As a result, it is sufficient to establish that for every $\xi\in E_h$, one has $\ell_g>\ell_{h_\xi}^\xi$. Let $\xi \in E_h$. As $d(0,h)=d(0,h_\xi)$ and $g\fleche h$, by definition of $(G_\eta,S_\eta)$, one has $d(0,g)=d(0,g_\xi)=d(0,h)-1$. Thus, one has $g_\xi\fleche h_\xi$ and $\ell_g \geq \ell_{g_\xi}^\xi>\ell_{h_\xi}^\xi$. As a result, the inequality $\ell_g \geq \sup^+\{\ell_h:g\fleche h\}$ holds.

Let us now establish that $\ell_g \leq \sup^+\{\ell_h:g\fleche h\}$. Let thus $\xi\in [0,\eta)$ be such that $d(0,g_\xi)=d(0,g)$ and let us prove that $\ell_{g_\xi}^\xi\leq \sup^+\{\ell_h:g\fleche h\}$. Recall that $\ell_{g_\xi}^\xi=\sup^+\{\ell_{\tilde h}^\xi:g_\xi \fleche \tilde h\in G_\xi\}$. By definition of $\ell_h$ and as $g_\xi \fleche {\tilde h} \implies \exists h\in G_\eta,~h_\xi={\tilde h}~\&~g\fleche h$, one has $\sup^+\{\ell_{\tilde h}^\xi: g_\xi \fleche \tilde h\in G_\xi\}\leq \sup^+\{\ell_h:g\fleche h\}$. As a result, the inequality $\ell_{g_\xi}^\xi\leq \sup^+\{\ell_h:g\fleche h\}$ holds, and the proof is complete.
\end{proof}

\begin{prop}
Let $\mathcal{G}$ be a Cayley graph of an abelian group. Assume that $\mathcal{G}$ contains a geodesic ray. Then, $\mathcal{G}$ contains a bi-infinite geodesic ray.
\end{prop}

\begin{proof}
Let $\gamma$ denote a geodesic ray. For $n\in \N$, set $s_n:=\gamma(n+1)-\gamma(n)$. Define $\kappa:\Z \to V$ as follows:
\begin{itemize}
\item for every $n\geq 0$, set $\kappa(n):=\sum_{0\leq k <n} s_{2k}$;
\item for every $n<0$, set $\kappa(n):=\sum_{0\leq k <-n} s_{2k+1}$.
\end{itemize}
Note that $\kappa$ is a path. For every $n\in \N$, one has
$$
\kappa(n)-\kappa(-n)=\sum_{0\leq k \leq 2n-1} s_k=\gamma(2n)-\gamma(0).
$$
As $\gamma$ is geodesic, the path $\kappa$ is also geodesic.
\end{proof}

For further use, say that a  path $\chemin$ is \defini{\zigzagfree~} if for every $k$ such that $\chemin(k)$ and $\chemin(k+2)$ are both defined, there is a unique path of length 2 connecting $\chemin(k)$ and $\chemin(k+2)$.

\begin{prop}
Lemma~\ref{closedlemma} does not hold if Condition~\ref{locfi} is removed, even if $\mathcal{G}$ is assumed to be a Cayley graph of an abelian group and $\mathfrak{P}$ is assumed to be invariant under $\chemin \mapsto (k\mapsto \chemin(-k))$.
\end{prop}

\vspace{0.4 cm}

\begin{proof}
Let $\mathcal{G}$ denote the Cayley graph of $G := \bigoplus_{n\geq 2} \Z/n\Z$ relative to the set consisting of the elements of $G$ of which all entries are zero but precisely one which is equal to 1. Label each edge by the index of the coordinate where its extremities disagree. This labelling is preserved by every automorphism of $\mathcal{G}$, as for any $n\geq 3$, an edge is labelled $n$ if and only if there is a \zigzagfree~path of length $n-1$ connecting its extremities --- and an edge has label 2 if and only if its label is not larger than or equal to 3. Denote the label of an edge $e$ by $N_e$.

Let $\mathfrak{P}$ denote the set of the paths $\chemin$ such that, for every $k$ such that $\chemin(k)$ and $\chemin(k+3)$ are both defined, one has $|N_{\{\chemin(k),\chemin(k+1)\}}-N_{\{\chemin(k+1),\chemin(k+2)\}}|=1$ and $N_{\{\chemin(k),\chemin(k+1)\}}\not= N_{\{\chemin(k+2),\chemin(k+3)\}}$. Since the labelling is preserved by every automorphism, Condition~\ref{condinv} of Lemma~\ref{closedlemma} is satisfied. By definition, $\mathfrak{P}$ is invariant under $\chemin\mapsto (k\mapsto \chemin(\pm k +  k_0))$. Condition~\ref{condclos} holds because $\mathfrak{P}$ is defined by a conjunction of conditions each of which involves finitely many (here four) coordinates. The set $\mathfrak{P}$ contains the infinite path $\N\to G$ mapping $k$ to $(\mathds{1}_{n\leq k+1})_n$, so that Condition~\ref{condlong} also holds.
However, $\mathfrak{P}$ cannot contain a bi-infinite path. Indeed, let $\chemin \in \mathfrak{P}$ be bi-infinite. The map $k\mapsto N_{\{\chemin(k),\chemin(k+1)\}}$ is well-defined from $\Z$ to $\N$ and of the form $k\mapsto ak+b$ for some $a\in\{-1,1\}$ and some $b\in \N$: such a map is contradictory, which ends the proof.
\end{proof}


\section{Weak isomorphy and transitivity}

\begin{prop}\label{nonhausdorff}
Let $G := \bigoplus_{n\geq 2} \Z/n\Z$, and let $S$ denote the (generating) subset of $G$ consisting of the elements of which all entries are zero but precisely one which is equal to 1. Let $G'$ denote $G\oplus \Z$, and let $S'$ denote $(S\times \{0\})\cup \{(0,1)\}$, which is a generating subset of $G'$. Then the Cayley graphs of $(G,S)$ and $(G',S')$ are weakly isomorphic but are not isomorphic.
\end{prop}

\begin{proof}
Let $(G,S)$ and $(G',S')$ be as in the statement of the proposition. Let $\mathcal{G}$ denote $\Cay(G,S)$ and $\mathcal{G}'$ denote $\Cay(G',S')$.

First, let us prove that $\mathcal{G}$ and $\mathcal{G}'$ are weakly isomorphic. Let $G_0$ denote $\bigoplus_{n\geq 2} \Z$, and let $S_0$ be the (generating) subset of $G$ consisting of the elements of which all entries are zero but precisely one which is equal to 1. Let $H:=\bigoplus_{n\geq 2} n\Z \subset G_0$. One can realise $(G,S)$ as $(G_0/H,\overline{S_0})$. Let $N\geq 0$, and let $H_N:=\bigoplus_{n\geq 2} k_{N,n}\Z \subset G_0$, where $k_{N,n}$ is equal to $n$ if $n<2N+2$, 0 if $n=2N+2$ and $n-1$ otherwise. One can realise $(G',S')$ as $(G_0/H_N,\overline{S_0})$. As $H_N$ and $H$ have the same intersection with the ball of radius $2N+1$ and centre 0 in $\Cay(G_0,S_0)$, the transitive graphs $\mathcal{G}$ and $\mathcal{G}'$ have isomorphic balls of radius $N$. Since $N$ is arbitrary, $\mathcal{G}$ and $\mathcal{G}'$ are weakly isomorphic.

Now, let us show that $\mathcal{G}$ and $\mathcal{G}'$ are not isomorphic. Recall that we say that a bi-infinite geodesic path $\gamma$ is \zigzagfree~if for every $n\in \Z$, there is a unique geodesic path between $\gamma(n)$ and $\gamma(n+2)$.
Since $\mathcal{G}$ does not admit a \zigzagfree~bi-infinite geodesic path while $\mathcal{G}'$ does, these graphs cannot be isomorphic.
\end{proof}

\begin{prop}
There is a graph that is weakly transitive and weakly isomorphic to a transitive graph, but not transitive.
\end{prop}

\begin{proof}
We use the notation of the proof of Proposition~\ref{nonhausdorff}. Let $V$ denote the set of non-empty finite words over the alphabet $G\cup G'$ such that the following two conditions hold:
\begin{itemize}
\item the first letter and only the first letter belongs to $G$;
\item only the first and the last letter of the word are allowed to be equal to $0_G$ or $0_{G'}$.
\end{itemize}
This is the vertex-set. Two vertices are declared to be adjacent if and only if one of the following (incompatible) conditions holds:
\begin{enumerate}
\item \label{ordredeux} one word can be written as the other one followed by $0_{G'}$;
\item \label{ggprime} one word can be written as the other one with the last letter modified by adding to it an element of $S$ or $S'$.
\end{enumerate}
This ``graph'' indeed is non-empty and connected.
As $\mathcal{G}$ and $\mathcal{G}'$ are weakly transitive, the graph $\mathcal{H}$ we have built is weakly transitive, and weakly isomorphic to the free product of $\mathcal{G}$ (or $\mathcal{G}'$) with the graph $K_2=(\{0,1\},\{\{0,1\}\})$. In particular, it is weakly transitive and weakly isomorphic to a transitive graph.

Let us now show that $\mathcal{H}$ is not transitive. Notice that removing an edge of $\mathcal{H}$ leaves its endpoints in different connected components if and only if it exists due to Condition~\ref{ordredeux}. As a result, a graph automorphism of $\mathcal{H}$ must map every connected component of $\mathcal{H}'$ onto a connected component of $\mathcal{H}'$, where $\mathcal{H}'$ stands for the ``graph'' with vertex-set $V$ and where only the edges due to Condition~\ref{ggprime} are kept. As the connected component of $[0_G]$ is isomorphic to $\mathcal{G}$ and that of $[0_G;0_{G'}]$ is isomorphic to $\mathcal{G}'$, and as $\mathcal{G}$ and $\mathcal{G}'$ are not isomorphic, no automorphism of $\mathcal H$ maps $[0_G]$ to $[0_G;0_{G'}]$. Thus, $\mathcal{H}$ is not transitive.
\end{proof}

\paragraph{Acknowledgements.} I would like to thank the Weizmann Institute of Science and my postdoctoral hosts --- Itai Benjamini and Gady Kozma --- for the excellent working conditions they have provided to me. I am also grateful to my postdoctoral hosts and Emmanuel Jacob for interesting discussions.

\begin{small}
\newcommand{\etalchar}[1]{$^{#1}$}

\end{small}
\end{document}